\documentclass[12pt,fleqn]{article}

\usepackage{amsmath,amssymb,amsthm}
\usepackage{geometry} 
\usepackage[pdfpagemode=UseNone,pdfstartview=FitH]{hyperref}
\usepackage{color,tikz-cd,float}

\newcommand{\bgset}[1]{\big\{#1\big\}}

\newcommand{\F}{{\mathcal F}}
\newcommand{\M}{{\mathcal M}}
\newcommand{\N}{\mathcal N}
\newcommand{\norm}[2][]{\left\|#2\right\|_{#1}}

\newcommand{\PS}[1]{$(\text{PS})_{#1}$}

\newcommand{\R}{\mathbb R}
\newcommand{\restr}[2]{\left.#1\right|_{#2}}
\newcommand{\seq}[1]{\left(#1\right)}
\newcommand{\set}[1]{\left\{#1\right\}}

\newcommand{\Z}{\mathbb Z}

\newenvironment{enumroman}{\begin{enumerate}

}{\end{enumerate}}

\newtheorem{lemma}{Lemma}[section]
\newtheorem{proposition}[lemma]{Proposition}
\newtheorem{theorem}[lemma]{Theorem}
\newtheorem{corollary}[lemma]{Corollary}

\theoremstyle{definition}

\theoremstyle{remark}
\newtheorem{remark}[lemma]{Remark}

\numberwithin{equation}{section}

\title{On the $N$-dimensional Schr\"{o}dinger--Poisson--Slater equation à la Brezis--Nirenberg  \bf}
\author{\bf Kanishka Perera\\
Department of Mathematics\\
Florida Institute of Technology\\
150 W University Blvd, Melbourne, FL 32901-6975, USA\\
\em kperera@fit.edu\\
[\medskipamount]
\bf Kaye Silva\\
Instituto de Matem\'{a}tica e Estat\'{i}stica\\
Universidade Federal de Goi\'{a}s\\
Rua Samambaia, 74001-970 Goi\^{a}nia, GO, Brazil\\
\em kayesilva@ufg.br}
\date{}

\begin{document}

\maketitle

\begin{abstract}
With aid of the Pohozaev's identity and Nehari manifold, we prove the existence and multiplicity of solutions to $N$-dimensional Schr\"{o}dinger--Poisson--Slater type equations involving critical exponents, by considering prescribed energy solutions. 
\end{abstract}

\begin{center}
	\begin{minipage}{12cm}
		\tableofcontents
	\end{minipage}
\end{center}

\newpage

\section{Introduction}
In this work we are interested in the following Schr\"{o}dinger--Poisson--Slater type equation
\begin{equation} \label{103}
	- \Delta u +\left(I_\alpha \star |u|^p\right) |u|^{p-2}u = \lambda\, |u|^{r-2} u + |u|^{2^*-2}\, u \quad \text{in } \R^N,
\end{equation}
where $N\ge 3$,  $\alpha\in (0,N)$, $1< p< (N+\alpha)/(N-2)$, $r>1$, $2^*=2N/(N-2)$, $\lambda>0$ is a real parameter and $I_\alpha:\mathbb{R}^N\to \mathbb{R}$ is the Riesz potential of order $\alpha$, defined for $x\in \mathbb{R}^N\setminus\{0\}$ as 
\begin{equation*}
	I_\alpha(x)=\frac{A_\alpha}{|x|^{N-\alpha}},\ \ \ A_\alpha=\frac{\Gamma(\frac{N-\alpha}{2})}{\Gamma(\frac{\alpha}{2})\pi^{N/2}2^\alpha}.
\end{equation*}

When $p=\alpha=2$ and $N=3$, equation \eqref{103} has the form 

\begin{equation} \label{10}
	- \Delta u + \left(I_2\star u^2\right) u = h(u) \quad \text{in } \R^3.
\end{equation}
 This equation is related to the Thomas--Fermi--Dirac--von\;Weizs\"acker model in Density Functional Theory (DFT), where the local nonlinearity $h$ takes the form $h(u) = u^{5/3} - u^{7/3}$ (see, e.g., Lieb \cite{MR629207,MR641371}, Le Bris and Lions \cite{MR2149087}, Lu and Otto \cite{MR3251907}, Frank et al.\! \cite{MR3762278}, and references therein). More general local nonlinearities also arise in DFT and quantum chemistry models such as Kohn-Sham's, where $h$ is the so-called exchange-correlation potential and is not explicitly known (see, e.g., Anantharaman and Canc\`es \cite{MR2569902} and references therein). 
 
  The $N$-dimensional version of \eqref{10} was proposed by Bao et al. \cite{BaMaSt} and recently, in Mercuri et al. \cite{MR3568051}, a thorough study of 
 
 \begin{equation*} 
 	- \Delta u +\left(I_\alpha \star |u|^p\right) |u|^{p-2}u =  |u|^{r-2} u\ \text{in } \R^N,
 \end{equation*}
 was accomplished (see also Lions \cite{MR636734} and Ruiz \cite{MR2679375}).

  The purpose of this paper is to study the existence and multiplicity of solutions with prescribed energy for equation \eqref{103} that can be viewed as a subcritical perturbation of
  
  \begin{equation*} 
  	- \Delta u +\left(I_\alpha \star |u|^p\right) |u|^{p-2}u = |u|^{2^*-2}\, u \quad \text{in } \R^N,
  \end{equation*}
which is an equation in the spirit of Brezis and Nirenberg \cite{BrNi}. The case $\alpha=p=2$, $N=3$ of equation \eqref{103} was already studied in Liu et al. \cite{LiZhHu} where existence of a single positive solution was show and in Mercuri and Perera \cite[Section 1.3.2]{MePe2} where multiplicity of solutions were considered.

We look for solutions of \eqref{103} over $E_R^{\alpha,p}(\mathbb{R}^N)$ whose definition is given as follows (see \cite[Section 2]{MR3568051}). The Coulomb space $\mathcal{Q}^{\alpha,p}(\mathbb{R}^N)$ is the vector space of measurable functions $u:\mathbb{R}^N\to \mathbb{R}$ such that
\begin{equation*}
	\|u\|_{\mathcal{Q}^{\alpha,p}}=\left(\int |I_{\alpha/2}\star |u|^p|^2\right)^{\frac{1}{2p}}<\infty.
\end{equation*}
The Coulomb-Sobolev space $E^{\alpha,p}(\mathbb{R}^N)$ is a subspace of functions $u\in \mathcal{Q}^{\alpha,p}(\mathbb{R}^N)$ such that $u$ is weakly differentiable in $\mathbb{R}^N$, $Du\in L^2(\mathbb{R}^N)$ and 

\begin{equation*}
	\|u\|=\left(\int |\nabla u|^2+\left(\int |I_{\alpha/2}\star |u|^p|^2\right)^{1/p}\right)^{1/2}<\infty. 
\end{equation*}
It follows that $E^{\alpha,p}(\mathbb{R}^N)$ is a uniformly convex, reflexive, Banach space with respect to the norm $\|u\|$ (see \cite[Proposition 2.2 and Proposition 2.9]{MR3568051}). We define $E_R^{\alpha,p}(\mathbb{R}^N)$ to be the subspace of radial functions in $E^{\alpha,p}(\mathbb{R}^N)$. Since it is a closed subspace, it is also a uniformly convex, reflexive, Banach space with respect to the norm $\|u\|$.

For simplicity we introduce the following notation
\begin{equation*}
	s_q=\frac{p(2-N)+\alpha+N}{p-1}, \ \ \ 	q=2\frac{2p+\alpha}{2+\alpha}, \ \ \ 	\sigma=\frac{2+\alpha}{2(p-1)},
\end{equation*}
and
\begin{equation*}
s_r=\sigma r -N.
\end{equation*}

 One can easily show that $s_q=\sigma q-N$. Moreover $q$ corresponds to the smaller exponent for which the continuous embedding $E^{\alpha,p}(\mathbb{R}^N)\hookrightarrow L^q(\mathbb{R}^N)$ holds true (see \cite[Theorem 1]{MR3568051}), so $E^{\alpha,p}(\mathbb{R}^N)\hookrightarrow L^\ell(\mathbb{R}^N)$ for all $\ell\in [q,2^*]$. 
 
 Let us show that \eqref{103} has a variational structure. For $u\in E_R^{\alpha,p}(\mathbb{R}^N)$ define
\begin{equation*}
	I(u)=\frac{1}{2}\int |\nabla u|^2+\frac{1}{2p}\int |I_{\alpha/2}\star |u|^p|^2,
\end{equation*}
\begin{equation*}
	F(u)=\frac{1}{r}\int |u|^{r},
\end{equation*}
\begin{equation*}
	G(u)=\frac{1}{2^*}\int |u|^{2^*}, 
\end{equation*}
and
\begin{equation*}
	\Phi_\lambda(u)=I(u)-\lambda F(u)-G(u), u\in E_R^{\alpha,p}(\mathbb{R}^N).
\end{equation*}
Then, solutions of \eqref{103} are critical points of the $C^1$ functional $\Phi_\lambda$. Note that, if 
\begin{equation*}
	u_t(x)=t^\sigma u(tx), u\in E_R^{\alpha,p}(\mathbb{R}^N)\setminus\{0\}, t>0,
\end{equation*} 
then
\begin{equation}\label{scaling}
	I(u_t)=t^{s_q}I(u),\ \ F(u_t)=t^{s_r}F(u),\ \ G(u_t)=t^{s_{2^*}}F(u).
\end{equation}

Given $c \in \R$, we are interested in finding pairs $(\lambda,u) \in \R \times E_R^{\alpha,p}(\mathbb{R}^N)$ such that $\Phi_\lambda'(u) = 0$ and $\Phi_\lambda(u) = c$. For $u \ne 0$, the equation $\Phi_\lambda(u) = c$ has the unique solution for $\lambda$ given by
\[
\lambda_c(u) = \frac{I(u) - G(u) - c}{F(u)}, \quad u \in E_R^{\alpha,p}(\mathbb{R}^N) \setminus \set{0}.
\]
For $c \in \R$, $u \in E_R^{\alpha,p}(\mathbb{R}^N) \setminus \set{0}$ is a critical point of $\Phi_\lambda$ with critical value $c$ if and only if it is a critical point of $\lambda_c$ with critical value $\lambda$ (see Proposition \ref{Proposition 7}). We will use this observation to study solutions of the equation \eqref{103} with energy $c$.

We will use a Nehari type manifold to study critical points of the functional $\lambda_c$. However, the standard Nehari manifold (see, e.g., Szulkin and Weth \cite{MR2768820}) is not suitable for this purpose because of the form of $\lambda_c$. Therefore we introduce a new manifold based on the following scaling property \eqref{scaling}:

 Set
\[
\varphi_{c,u}(t) = \lambda_c(u_t), \quad u \in E_R^{\alpha,p}(\mathbb{R}^N)\setminus\{0\},\, t > 0
\]
and let
\[
\N_c = \set{u_t : u \in E_R^{\alpha,p}(\mathbb{R}^N)\setminus\{0\},\, t > 0,\, \varphi_{c,u}'(t) = 0}.
\]
We will show that $\N_c$ is a $C^1$-Finsler manifold and a natural constraint for $\lambda_c$ (see Lemma \ref{NEHARI} and Lemma \ref{Lemma 6}). We will refer to it as the scaled Nehari manifold and study critical points of $\restr{\lambda_c}{\N_c}$ using a variational framework for scaled functionals recently developed in Mercuri and Perera \cite{MePe2}. We note here that the results of Leite et al. \cite{LeQuSi} cannot be applied in our context since their fibering maps are different.

We define the best Sobolev embedding constant. It is the biggest constant satisfying the inequality
\begin{equation*}
	S\|u\|_{2^*}^2\le \|\nabla u\|_2^2, u\in E_r^{\alpha,p}(\mathbb{R}^N).
\end{equation*} 
In what follows, $\lambda_k$ is the sequence of eigenvalues of the eigenvalue problem \eqref{1} given in Theorem \ref{Theorem 7} in the next section.
\begin{theorem}\label{thm1} Suppose $r\in [q,2^*)$ and
	\begin{equation*}
	c^*=\frac{1}{2}\frac{s_{2^*}-s_r}{s_{2^*}}\left(\frac{2^*}{2}\frac{s_q}{s_{2^*}}\right)^{\frac{s_q}{s_{2^*}-s_r}}S^{\frac{s_{2^*}}{s_{2^*}-s_r}}.
	\end{equation*} 
	Then, there exists a sequence $\tilde{\lambda}_k$ satisfying $\tilde{\lambda}_k\ge 0$, $\tilde{\lambda}_k\le \tilde{\lambda}_{k+1}$ for all $k\in \mathbb{N}$, and such that:
	\begin{enumroman}
		\item Equation \eqref{103} has no nontrivial solution with energy $c \le 0$. For each $c \in (0,c^*)$, there exists a sequence $\lambda_{c,k} \nearrow \infty$ as $k\to \infty$ such that equation \eqref{103} with $\lambda = \lambda_{c,k}$ has a nontrivial solution $u_{c,k}$ with $\Phi_{\lambda_{c,k}}(u_{c,k}) = c$.
		\item If $r>q$ and  $\lambda> \tilde{\lambda}_k$, then equation \eqref{103} has at least $k$ solutions.
		\item If $r=q$, $\alpha>1$ and $\lambda\in (\tilde{\lambda}_k,\lambda_k)$, then equation \eqref{103} has at least one solution.
		\item $\tilde{\lambda}_1=0$ in the following cases:
			\begin{equation}\label{cases}
			\begin{cases}
			\frac{2Np}{N+\alpha} > \frac{N}{N-2}, N+\alpha-(N-2)(p+1)\ge 0, (N-2)r-4>0,\\
			\frac{2Np}{N+\alpha} > \frac{N}{N-2}, N+\alpha-(N-2)(p+1)< 0, 2\alpha+(N-2)(r-2p)>0, \\
			\frac{2Np}{N+\alpha} = \frac{N}{N-2},4+\alpha-N>0, (N-2)r-4>0, \\
			\frac{2Np}{N+\alpha} = \frac{N}{N-2}, 4+\alpha-N\le 0, \alpha-N+(N-2)r>0, \\
			\frac{2Np}{N+\alpha} < \frac{N}{N-2}, (N-2)r-4>0. \\
		\end{cases}
		\end{equation}
	\end{enumroman}
\end{theorem}

We note here that, in the cases \eqref{cases}, Theorem \ref{thm1} guarantees the existence of solutions for all $\lambda>0$, if $r>q$, and for small $\lambda$, if $r=q$. Moreover, when $c>c^*$, the infimum of $\lambda_c$ over $\mathcal{N}_c$ is negative (see Proposition \ref{lbb2}) and, since the Pohozaev's identity (see \eqref{POHOZAEV}) implies non-existence of positive solutions for $\lambda\le 0$, we are forced to conclude (from Proposition \ref{Proposition 7}) that infimum cannot be attained. See also Remark \ref{rnexistence} for a brief discussion on non-existence of solutions when \eqref{cases} are not satisfied.

It is also worth noting that the technique used here can be used in many other types of equations involving critical terms in the spirit of Brezis--Nirenberg. In fact, it has already been used, in the context of a generalization of Szulkin and Weth \cite{MR2768820} results, to functionals of class $C^1(X\setminus\{0\})$ (see Leite et al. \cite[Section 4.3]{LeQuSi}). We also point out the work of Ramos Quoirin et al.\! \cite{MR4736027} and the references therein, where many types of prescribed energy equations are studied.

Unless otherwise stated all integrals are over $\mathbb{R}^N$ with respect to the Lebesgue measure. Moreover, when considering sequences, the passage to subsequences and its respective topological properties will be used without renaming it.

\section{Technical Results}

\subsection{Topological Properties of $\lambda_c$}
In this section we study the functional $\lambda_c$. From now on we assume that 
\begin{enumerate}
	\item[$(H_1)$] $r\in [q,2^*)$ and, if $r=q$, we also impose $\alpha>1$. 
\end{enumerate}
Note that, since $s_q>0$, then $s_r>0$ for all $r\in [q,2^*]$. Moreover $s_{r_2}-s_{r_1}=\sigma (r_2-r_1)>0$ for all $r_1,r_2\in [q,2^*]$ with $r_2>r_1$. This will be used throughout the paper. 

Given $c\in \mathbb{R}$ note that for all $u\in W\setminus\{0\}$ the equation $\Phi_\lambda(u)=c$ has a unique solution with respect to $\lambda$, given by
\begin{equation*}
	\lambda_c(u)=\frac{I(u)-G(u)-c}{F(u)}.
\end{equation*}
\begin{proposition} \label{Proposition 7}
	For $c \in \R$, $u \in E_R^{\alpha,p}(\mathbb{R}^N)\setminus\{0\}$ is a critical point of $\Phi_\lambda$ with critical value $c$ if and only if it is a critical point of $\lambda_c$ with critical value $\lambda$.
\end{proposition}

\begin{proof}
	We have
	\begin{equation} \label{25}
		\lambda_c(u) = \frac{\Phi_\lambda(u) - c}{F(u)} + \lambda
	\end{equation}
	and
	\begin{equation} \label{9}
		\lambda_c'(u) = \frac{\Phi_\lambda'(u) - (\lambda_c(u) - \lambda)\, F'(u)}{F(u)}.
	\end{equation}
	So $\Phi_\lambda(u) = c$ and $\Phi_\lambda'(u) = 0$ if and only if $\lambda_c(u) = \lambda$ and $\lambda_c'(u) = 0$.
\end{proof}

Note by \eqref{scaling} that
\begin{equation*}
	\varphi_{c,u}(t)=\frac{t^{s_q-s_r}I(u)-t^{s_{2^*}-s_r}G(u)-t^{-s_r}c}{F(u)},
\end{equation*}
and let 
\begin{equation*}
	\mathcal{N}_c=\{u_t: u\in E_R^{\alpha,p}(\mathbb{R}^N)\setminus\{0\}, t>0, \varphi_{c,u}'(t)=0\}.
\end{equation*}
We will also consider the following set 
\begin{equation*}
	\mathcal{M}=\{u\in E_R^{\alpha,p}(\mathbb{R}^N):I(u)=1\}.
\end{equation*}
\begin{proposition}\label{NEHARI} There holds:
	\begin{enumerate}
		\item[(i)] $\mathcal{M}$ is a $C^1$-Finsler manifold, which is complete, symmetric.
		\item[(ii)] 	$\mathcal{N}_c\neq\emptyset $ if, and only if, $c>0$. Moreover, for each $c>0$ and $u\in E_R^{\alpha,p}(\mathbb{R}^N)\setminus\{0\}$, the equation $\varphi_{c,u}'(t)=0$ has unique solution at $t:=t_c(u)$, which corresponds to a global maximizer of $\varphi_{c,u}$. Then $\mathcal{N}_c$ is a $C^1$-Finsler manifold, which is complete, symmetric and homeomorphic to $\mathcal{M}$ via an odd homeomorphism and can also be represented as
		\begin{eqnarray*}
			\mathcal{N}_c&=&\{u\in W\setminus\{0\}:\varphi_{c,u}'(1)=0 \}\\
			&=& \{u\in W\setminus\{0\}: H(u):=(s_r-s_q)I(u)+(s_{2^*}-s_r)G(u)=s_rc\},
		\end{eqnarray*}
		or 
		\begin{equation*}
			\mathcal{N}_c=\{u_{t_c(u)}: u\in \mathcal{M}\}.
		\end{equation*}
	\end{enumerate}

\end{proposition}
\begin{proof} $(i)$ This is clear since $I$ is continuous and even, and $I'(u)=0$ if, and only if, $u=0$. 
	
	$(ii)$ That $\mathcal{N}_c\neq\emptyset $ if, and only if, $c>0$ is obvious. It is also clear that for each $c>0$ and $u\in E_R^{\alpha,p}(\mathbb{R}^N)\setminus\{0\}$, the equation $\varphi_{c,u}'(t)=0$ has unique solution at $t:=t_c(u)$, which corresponds to a global maximizer of $\varphi_{c,u}$, therefore it follows that
	\begin{equation*}
		t_c(u_t)=\frac{1}{t}t_c(u), u\in E_R^{\alpha,p}(\mathbb{R}^N)\setminus\{0\}, t>0,
	\end{equation*}
	which implies that $t_c(t_c(u)u)=1$ and hence the desired equalities concerning $\mathcal{N}_c$. Now observe that $u=0$ is the only critical point of $H$, so 	$\mathcal{N}_c$ is a $C^1$-Finsler manifold and since $H$ is continuous and even, $\mathcal{N}_c$ is complete and symmetric.  To conclude the proof we note that if $u\in E_R^{\alpha,p}(\mathbb{R}^N)\setminus\{0\}$, then $\pi(u)=u_{I(u)^{\frac{-1}{s_q}}}\in \mathcal{M}$ is a continuous projection from $E_R^{\alpha,p}(\mathbb{R}^N)\setminus\{0\}$ to $\mathcal{M}$ that, when restricted to $\mathcal{N}_c$, gives us the desired homeomorphism.
\end{proof}
In order to proceed we need to establish a Pohozaev's identity. In fact, as in \cite[Proposition 5.5]{MR3568051}, one can show that if $\Phi'_\lambda(u)=0$, then
\begin{equation}\label{POHOZAEV}
P(u):=	\frac{N-2}{2}\int |\nabla u|^2+\frac{N+\alpha}{2p}\int |I_{\alpha/2}\star |u|^p|^2-\frac{N\lambda}{r}\int |u|^r-\frac{N}{2^*}\int |u|^{2^*}=0.
\end{equation}
Since $\Phi_\lambda'(u)=0$, we can compute $\sigma \Phi_\lambda'(u)-P(u)=0$ to conclude that
\begin{equation}\label{pohozaev}
	\frac{s_q}{2}\int |\nabla u|^2+\frac{s_q}{2p}\int |I_{\alpha/2}\star |u|^p|^2-\frac{s_r\lambda}{r}\int |u|^r-\frac{s_{2^*}}{2^*}\int |u|^{2^*}=0.
\end{equation}

The importance of $\mathcal{N}_c$ comes from 
\begin{lemma} \label{Lemma 8}
	If $u \in E_R^{\alpha,p}(\mathbb{R}^N) \setminus \set{0}$ is a solution of the equation \eqref{103} with $\Phi_\lambda(u) = c$, then $u \in \N_c$.
\end{lemma}

\begin{proof}
	We have by \eqref{pohozaev} that
	\[
	s_q\, I(u) -s_r\lambda F(u)-s_{2^*}G(u)=0.
	\]
Since $I(u) - \lambda F(u) - G(u) = c$, so $H(u)= 0$ and by Proposition \ref{NEHARI}, $u\in \mathcal{N}_c$.
\end{proof}

For $c>0$ we define
\begin{equation*}
	\Lambda_c=(\lambda_{c})_{|\mathcal{N}_c}.
\end{equation*}

The next result shows that $\mathcal{N}_c$ is a natural constraint to $\lambda_c$.
\begin{lemma}\label{Lemma 6}
	If $\Lambda_c'(u)=0$, then $\lambda_c'(u)=0$.
\end{lemma}
\begin{proof} If $\Lambda_c'(u)=0$, then there exists $\mu\in \mathbb{R}$ such that $\lambda_c'(u)=\mu H'(u)$. Therefore
	\begin{equation*}
		I'(u)-G'(u)-\lambda_c(u)F'(u)=\mu F(u)H'(u).
	\end{equation*}
	By the Pohozaev's identity \eqref{pohozaev} we conclude that
	\begin{equation*}
		s_qI(u)-s_{2^*}G(u)-s_r\left(\frac{I(u)-G(u)-c}{F(u)}\right)F(u)=\mu F(u)[s_q(s_r-s_q)I(u)+s_{2^*}(s_{2^*}-s_r)G(u)],
	\end{equation*}
 which implies
	\begin{equation*}
		0=-H(u)+s_rc=\mu F(u)[s_q(s_r-s_q)I(u)+s_{2^*}(s_{2^*}-s_r)G(u)],
	\end{equation*}
	and thus $\mu=0$.
\end{proof}

\begin{lemma}\label{lbound}  For each $[a,b]\subset (0,\infty)$, there exist constants $C_1,C_2>0$ such that
	\begin{enumerate}
		\item[(i)] If $r\in(q,2^*)$, then $C_1\le \|u\|\le C_2$ for all $u\in \mathcal{N}_c$, $c\in [a,b]$.
		\item[(ii)] If $r=q$, then $C_1\le \|u\|, G(u)$ and $G(u)\le C_2$ for all $u\in \mathcal{N}_c$, $c\in [a,b]$. Moreover, if $\alpha>1$, then the functional $\Lambda_c$ is coercive for all $c>0$.	 
	\end{enumerate}

\end{lemma}
\begin{proof} The inequalities in $(i)$ and $(ii)$ are clear since $H(u)=s_rc$ for all $u\in \mathcal{N}_c$, $c\in [a,b]$. Now we prove that $\Lambda_c$ is coercive when $r=q$. Let $\seq{u_n}$ be a sequence in $\N_c$ such that $\norm{u_n} \to \infty$. Set
	\[
	t_n = I(u_n)^{-1/s_q}, \quad \widetilde{u}_n = (u_n)_{t_n}, \quad \widetilde{t}_n = t_n^{-1} = I(u_n)^{1/s_q}
	\]
Then $\widetilde{u}_n \in \M$ and
	\begin{equation} \label{27}
		u_n = (\widetilde{u}_n)_{\widetilde{t}_n}
	\end{equation}
 Since $\M$ is a bounded manifold, $\seq{\widetilde{u}_n}$ is bounded and hence converges weakly to some $\widetilde{u} \in E_R^{\alpha,p}(\mathbb{R}^N)$. Since $\norm{u_n} \to \infty$, $I(u_n) \to \infty$ and hence $\widetilde{t}_n \to \infty$. By $H(u_n)=0$ we have 
	\begin{equation} \label{30}
s_{2^*}\widetilde{t}_n^{s_{2^*}}G(\widetilde{u}_n)=s_rc, \forall n\in \mathbb{N}.
	\end{equation}
	Since $\widetilde{t}_n \to \infty$ we conclude that $G(\widetilde{u}_n) \to 0$. Once $\alpha>1$ and 
	\begin{equation*}
		q>\frac{2[2p(N-1)+N-\alpha]}{3N+\alpha-4}:=q_0, 
	\end{equation*}
	we can choose $q_1\in (q_0,q)$ and $\theta$ satisfying $1/q=(1-\theta)/q_1+\theta/2^*$ such that
	\begin{equation*}
		\|\widetilde{u}_n\|_{q}\le \|\widetilde{u}_n\|_{q_1}^{1-\theta}\|\widetilde{u}_n\|_{2^*}^\theta, \forall n\in \mathbb{N}.
	\end{equation*}
Because $\|\widetilde{u}_n\|_{q_1}$ is bounded (see \cite[Theorem 4]{MR3568051}) we conclude that $\|\widetilde{u}_n\|_{q}\to 0$  and hence $F(\widetilde{u}_n)\to 0$ as $n\to \infty$. Since (because $H(u_n)=0$)
	\[
	\Lambda_{c}(u_n) = \frac{I(u_n)-G(u_n)-c}{F(u_n)}=\frac{I(u_n)-dc}{F(u_n)}=\frac{1-dc\widetilde{t}_n^{-s_q}}{F(\widetilde{u}_n)}, \forall n\in \mathbb{N},
	\]
 then $\Lambda_{c}(u_n) \to \infty$ and the proof is complete.
\end{proof}

Now we focus on showing that if $c\in (0,c^*)$, then  $\inf_{u\in \mathcal{N}_c}\Lambda_c(u)>0$. This will be used to prove the Palais--Smale condition. We note that imposing $\inf_{u\in \mathcal{N}_c}\Lambda_c(u)>0$ is not a restriction to our approach, since by the Pohozaev identity the equation \eqref{103} has no positive solutions when $\lambda\le 0$. The following  result is straightforward.

\begin{lemma}\label{lt0} Let
	\begin{equation*}
		z_u(t)=\frac{S}{2}\left(\int |u|^{2^*}\right)^{\frac{2}{2^*}}t^{s_q}-\frac{1}{2^*}\int |u|^{2^*}t^{s_{2^*}}, t>0.
	\end{equation*}
	Then, $z_u$ has a unique global maximizer at
	\begin{equation*}
		t_0(u)=\left[\frac{2^*}{2}\frac{s_q}{s_{2^*}}\frac{S}{\left(\int |u|^{2^*}\right)^{\frac{2}{N}}} \right]^{\frac{1}{s_{2^*}-s_q}}
	\end{equation*} 
	and 
	\begin{equation*}
		c^*:=z_u(t_0)=\frac{1}{2}\frac{s_{2^*}-s_q}{s_{2^*}}\left(\frac{2^*}{2}\frac{s_q}{s_{2^*}}\right)^{\frac{s_q}{s_{2^*}-s_q}}S^{\frac{s_{2^*}}{s_{2^*}-s_q}}.
	\end{equation*}
\end{lemma}
\begin{lemma}\label{ineinfipositive} There holds
		\begin{equation*}
		\Lambda_c(u)\ge t_0(u)^{-s_r}\frac{c^*-c}{F(u)} ,\forall u\in \mathcal{N}_c.
	\end{equation*}
\end{lemma}
\begin{proof} By Proposition \ref{NEHARI} item $(ii)$ we have that
	\begin{equation*}
		\Lambda_c(u)\ge\varphi_{c,u}(t)= \lambda_c(u_{t}),\forall u\in \mathcal{N}_c, t>0,
	\end{equation*}
	which implies that
	\begin{eqnarray*}
		\Lambda_c(u)&\ge& t^{-s_r}\frac{\left[\frac{S}{2}\left(\int |u|^{2^*}\right)^{\frac{2}{2^*}}+\frac{1}{2p}\int |I_{\alpha/2}\star |u|^p|^2\right]t^{s_q}-\frac{t^{s_{2^*}}}{2^*}\int |u|^{2^*}-c}{F(u)} \\
		&=& t^{-s_r} \frac{z_u(t)+\frac{t^{s_q}}{2p}\int |I_{\alpha/2}\star |u|^p|^2-c}{F(u)} ,\forall u\in \mathcal{N}_c, t>0,
	\end{eqnarray*}
	and hence, by Lemma \ref{lt0}, 
		\begin{equation*}
		\Lambda_c(u)\ge t_0(u)^{-s_r}\frac{c^*-c}{F(u)}, \forall u\in \mathcal{N}_c.
	\end{equation*}
\end{proof}
\begin{lemma}\label{lbb} If $c\in (0,c^*)$, then  $\inf_{u\in \mathcal{N}_c}\Lambda_c(u)>0$.
	
\end{lemma}
\begin{proof} We already know by Lemma \ref{ineinfipositive} that $\inf_{u\in \mathcal{N}_c}\Lambda_c(u)\ge0$. Let $u_n\in \mathcal{N}_c$ be a minimizing sequence to $\inf_{u\in \mathcal{N}_c}\Lambda_c(u)$.  First assume $r>q$. Since $H(u_n)=0$ for all $u\in \mathcal{N}_c$, we have that
\begin{equation}\label{gzero}
	\Lambda_c(u_n)=\frac{s_rc-(s_{2^*}-s_q)G(u_n)}{(s_r-s_q)F(u_n)}, \forall n\in \mathbb{N}.
\end{equation}

By Lemma \ref{lbound} we have that $F(u_n)$ is bounded. Moreover, if $G(u_n)\to 0$, then \eqref{gzero} implies that $\Lambda_c(u_n)$ is away from zero, while if $G(u_n)$ is away from zero, then Lemma \ref{ineinfipositive} implies that $\Lambda_c(u_n)$ is away from zero. In both cases we conclude  $\inf_{u\in \mathcal{N}_c}\Lambda_c(u)> 0$.

Now suppose $r=q$. By Lemma \ref{lbound} we know that $F(u_n)$ is bounded and $G(u_n)$ is away from zero, therefore by Lemma \eqref{ineinfipositive} we conclude that $\inf_{u\in \mathcal{N}_c}\Lambda_c(u)>0$ if $c\in (0,c^*)$.

\end{proof}
We shall see in the next section that $\inf_{u\in \mathcal{N}_c}\Lambda_c(u)<0$ when $c>c^*$. Now we show that $\Lambda_c$ satisfies the   \PS{} condition if $c\in (0,c^*)$. To this end let us first note an auxiliary result.

\begin{corollary}\label{cinfinity} Suppose $0<c<c^*$ and $u_n\in \mathcal{N}_c$. If  $F(u_n)\to 0$, then $\Lambda_c(u_n)\to \infty$.
	
\end{corollary}
\begin{proof} If $r>q$ it follows from \eqref{gzero} and Lemma \ref{ineinfipositive} as in the proof of Lemma \ref{lbb}. If $r=q$ the proof is straightforward from Lemmas \ref{lbound} and \ref{ineinfipositive}.
\end{proof}
\begin{lemma} If  $0<c<c^*$, then $\Lambda_c$ satisfies the {\em \PS{}} condition.
\end{lemma}

\begin{proof} Suppose that $u_n\in \mathcal{N}_c$ satisfies $\Lambda'_c(u_n)\to 0$ and $\Lambda_c(u_n)$ is bounded. We can find a sequence $\eta_n\in \mathbb{R}$ such that 
	\begin{equation*}
		\lambda_c'(u_n)=\eta_n H'(u_n)+o(1), \forall n\in \mathbb{N}.
	\end{equation*}
	By Lemma \ref{lbound} we have that $F(u_n)$ and $G(u_n)$ are bounded and hence
	\begin{equation}\label{99}
		I'(u_n)-G'(u_n)-\lambda_c(u_n)F'(u_n)=\eta_n F(u_n)[(s_r-s_q)I'(u_n)+(s_{2^*}-s_r)G'(u_n)]+o(1), \forall n\in \mathbb{N},
	\end{equation}
We claim that $\eta_n$ is bounded. On the contrary, by \eqref{99} we conclude that $F(u_n)[(s_r-s_q)I'(u_n)+(s_{2^*}-s_r)G'(u_n)]\to 0$ as $n\to \infty$. Since $u_n$ is bounded it follows that $F(u_n)[(s_r-s_q)I'(u_n)+(s_{2^*}-s_r)G'(u_n)]u_n\to 0$ as $n\to \infty$ and hence ($I(u_n),G(u_n)\to 0$ if $r>q$, $G(u_n)\to 0$ if $r=q$)  or $F(u_n)\to 0$. The first case is impossible by Lemma \ref{lbound} so $F(u_n)\to 0$ which implies, by Corollary \ref{cinfinity}, that $\Lambda_c(u_n)\to \infty$, a contradiction, therefore $\eta_n$ is bounded. Note that

	\begin{equation}\label{100}
		[1-(s_r-s_q)\eta_nF(u_n)]I'(u_n)=G'(u_n)+\lambda_c(u_n)F'(u_n)+(s_{2^*}-s_r)\eta_n F(u_n)G'(u_n)+o(1), \forall n\in \mathbb{N}.
	\end{equation}
	From now on we assume that $u_n \rightharpoonup u$, $\eta_n\to \eta$ and $\lambda_c(u_n)\to \lambda$ as $n\to \infty$. It follows that $G'(u_n)\to G'(u)$, $F'(u_n)\to F'(u)$. We claim that $1-(s_r-s_q)\eta_nF(u_n)$ is away from zero. Indeed, if $r=q$ this is trivial, so assume that $r>q$, then  $1-(s_r-s_q)\eta_nF(u_n)\to 1-(s_r-s_q)\eta F(u)=0$ as $n\to \infty$, then $\eta>0$. However \eqref{100} implies that
	\begin{equation*}
		G'(u)+\lambda F'(u)+(s_{2^*}-s_r)\eta F(u)G'(u)=0,
	\end{equation*} 
	and hence $\eta<0$, a contradiction. Therefore $1-(s_r-s_q)\eta_nF(u_n)$ is away from zero and by \eqref{100} we conclude that $I'(u_n)(u_n-u)\to 0$. Now observe that
	
	\begin{eqnarray*}
		o(1)=(I'(u_n)-I'(u))(u_n-u)&=&\int|\nabla u_n-\nabla u|^2+\int |I_{\alpha/2}\star |u_n|^p|^2 \\
		&-& \int (I_{\alpha/2}\star |u_n|^p)(I_{\alpha/2}\star |u_n|^{p-2}u_nu) \\  
			&-& \int (I_{\alpha/2}\star |u|^p)(I_{\alpha/2}\star |u|^{p-2}uu_n) \\  
				&+& \int |I_{\alpha/2}\star |u|^p|^2 \\ 
				&\ge & \int|\nabla u_n-\nabla u|^2+\frac{1}{p}\int (I_{\alpha/2}\star |u_n|^p -I_{\alpha/2}\star |u|^p)^2,
	\end{eqnarray*}
	where the equality comes from the semigroup property of $I_\alpha$ and the inequality comes from Young's inequality. Since $ E_R^{\alpha,p}(\mathbb{R}^N)$ is uniformly convex we conclude that $u_n\to u$ and the proof is complete.
\end{proof}
Let $\F_c$ denote the class of symmetric subsets of $\mathcal{N}_c$ and let $i(M)$ denote the $\Z_2$-cohomological index of $M \in \F_c$ (see Fadell and Rabinowitz \cite{MR0478189}). For $k \ge 1$, let $\F_{c,k} = \bgset{M \in \F_c : i(M) \ge k}$ and set
\begin{equation} \label{50}
	\lambda_{c,k} := \inf_{M \in \F_{c,k}}\, \sup_{u \in M}\, \Lambda_c(u).
\end{equation}
We have the following proposition (see Perera et al.\! \cite[Proposition 3.52]{MR2640827}).

\begin{proposition} \label{Proposition 3} If $c\in (0,c^*)$, then 
	$\lambda_{c,k} \nearrow \infty$ is a sequence of critical values of $\Lambda_c$.
\end{proposition}
\subsection{Estimates on Talenti Functions and Study of $\Lambda_{c}$, $c\ge c^*$}
As noted in the last section the inequality $\inf_{u\in \mathcal{N}_c}\Lambda_c(u)>0$ is not a restriction since equation \eqref{103} has no positive solutions when $\lambda\le 0$. In this section we show that for some cases, which depend on $\alpha,N,p$, we have that $\inf_{u\in \mathcal{N}_c}\Lambda_{c^*}(u)=0$. We also show that $\inf_{u\in \mathcal{N}_c}\Lambda_c(u)<0$ if $c>c^*$.

\begin{lemma}\label{ltalenti} Let $\rho>0$ and $\theta\in C_0^\infty(\mathbb{R}^N)$ such that $\operatorname{supp}\theta \subset B_{\rho}(0)$, $0\le \theta\le 1$, $\theta =1$ on $B_{\rho/2}(0)$ and $\theta$ is radial. Given $\varepsilon>0$ let
	\begin{equation*}
		u_\varepsilon(x)=\frac{\theta(x)\varepsilon^{\frac{N-2}{2}}}{(\varepsilon^2+|x|^2)^{\frac{N-2}{2}}}\ \ \ \mbox{and}\ \ \ v_\varepsilon(x)=\frac{u_\varepsilon(x)}{\|u_\varepsilon(x)\|_{2^*}}.
	\end{equation*}
	Then
	\begin{equation*}
		\|\nabla v_\varepsilon\|_2^2=S+O(\varepsilon^{N-2}),
	\end{equation*}
	\begin{equation*}
\|v_\varepsilon\|_\ell^\ell = 
	\begin{cases}
		O(\varepsilon^{(2N-(N-2)\ell)/2}) & \text{if } \ell > \frac{N}{N-2},\\
			O(\varepsilon^{N/2}|\log\varepsilon|)       & \text{if } \ell = \frac{N}{N-2}, \\
			O(\varepsilon^{(N-2)\ell/2})   & \text{if } \ell < \frac{N}{N-2},
	\end{cases}
	\end{equation*}
	\begin{equation*}
	\|v_\varepsilon\|_\ell^\ell\ge C \varepsilon^{(2N-(N-2)\ell)/2},\ \ell\in (2,2^*),
	\end{equation*}
	where $C>0$ is a constant. Finally
	\begin{equation*}
\int |I_{\alpha/2}\star |v_{\varepsilon}|^p|^2\le 
	\begin{cases}
		O(\varepsilon^{N+\alpha-(N-2)p}) & \text{if } \frac{2Np}{N+\alpha} > \frac{N}{N-2},\\
		O(\varepsilon^{(N+\alpha)/2}|\log\varepsilon|^{\frac{N+\alpha}{N}})       & \text{if }  \frac{2Np}{N+\alpha} = \frac{N}{N-2}, \\
		O(\varepsilon^{(N-2)p})   & \text{if }  \frac{2Np}{N+\alpha} < \frac{N}{N-2}.
	\end{cases}
\end{equation*}
\begin{proof} For the estimates concerning $	\|\nabla v_\varepsilon\|_2^2$ and $	\|v_\varepsilon\|_\ell^\ell$ see \cite[Lemma 1.1]{BrNi} and \cite[Proof of Lemma 3.4]{BaCoSeSo}. To prove the last estimate, note by the Hardy--Littlewood--Sobolev inequality that there exists $C>0$ such that
	\begin{equation*}
		\int |I_{\alpha/2}\star |v_{\varepsilon}|^p|^2\le C\left(\int |v_{\varepsilon}|^{\frac{2Np}{N+\alpha}}\right)^{\frac{N+\alpha}{N}}, u\in E_R^{\alpha,p}(\mathbb{R}^N),
	\end{equation*}
	and hence the proof is complete.
\end{proof}
\end{lemma}
Our goal now is to study $\Lambda_{c^*}(v_{\varepsilon})$, so the following estimates are in order.
\begin{lemma}\label{ll1} If $\varepsilon>0$ is small, then $	\frac{O(\varepsilon^{N-2})}{2}+\frac{1}{2p}\int |I_{\alpha/2}\star |v_{\varepsilon}|^p|^2\le $
	
			\begin{equation*}
		\begin{cases}
			O(\varepsilon^{N-2}) & \text{if }  \frac{2Np}{N+\alpha} > \frac{N}{N-2}\ \mbox{and}\ N+\alpha-(N-2)(p+1)\ge 0,\\
			O(\varepsilon^{N+\alpha-(N-2)p} )      & \text{if } \frac{2Np}{N+\alpha} > \frac{N}{N-2}\ \mbox{and}\ N+\alpha-(N-2)(p+1)< 0, \\
				O(\varepsilon^{N-2} )      & \text{if } \frac{2Np}{N+\alpha} = \frac{N}{N-2}\ \mbox{and}\ 4+\alpha-N>0, \\
				O(\varepsilon^{(N+\alpha)/2}|\log\varepsilon|^{\frac{N+\alpha}{N}})      & \text{if } \frac{2Np}{N+\alpha} = \frac{N}{N-2}\ \mbox{and}\ 4+\alpha-N\le 0, \\
						O(\varepsilon^{N-2} )      & \text{if } \frac{2Np}{N+\alpha} < \frac{N}{N-2}. \\
		\end{cases}
	\end{equation*}

\end{lemma}
\begin{proof} Straightforward from Lemma \ref{ltalenti} and the properties of $O$.
\end{proof}
Now the most important estimate:
%
\begin{lemma}\label{la1} If $\varepsilon>0$ is small, then $	\frac{	\frac{O(\varepsilon^{N-2})}{2}+\frac{1}{2p}\int |I_{\alpha/2}\star |v_{\varepsilon}|^p|^2}{\frac{1}{r}\int |v_\varepsilon|^r}\le $
	
	\begin{equation*}
		\begin{cases}
		O(\varepsilon^{[(N-2)r-4]/2}) & \text{if }  \frac{2Np}{N+\alpha} > \frac{N}{N-2}\ \mbox{and}\ N+\alpha-(N-2)(p+1)\ge 0,\\
		O(\varepsilon^{[2\alpha+(N-2)(r-2p)]/2} )      & \text{if } \frac{2Np}{N+\alpha} > \frac{N}{N-2}\ \mbox{and}\ N+\alpha-(N-2)(p+1)< 0, \\
		O(\varepsilon^{[(N-2)r-4]/2})     & \text{if } \frac{2Np}{N+\alpha} = \frac{N}{N-2}\ \mbox{and}\ 4+\alpha-N>0, \\
		O(\varepsilon^{[\alpha-N+(N-2)r]/2}|\log\varepsilon|^{\frac{N+\alpha}{N}})      & \text{if } \frac{2Np}{N+\alpha} = \frac{N}{N-2}\ \mbox{and}\ 4+\alpha-N\le 0, \\
		O(\varepsilon^{[(N-2)r-4]/2})     & \text{if } \frac{2Np}{N+\alpha} < \frac{N}{N-2}. \\
	\end{cases}
	\end{equation*}
	
\end{lemma}
\begin{proof} Straightforward from Lemma \ref{ll1}.
\end{proof}

\begin{proposition}\label{lbb1} $\inf_{u\in \mathcal{N}_{c^*}}\Lambda_{c^*}(u)=0$ in the following cases:
	\begin{equation*}
	\begin{cases}
 \frac{2Np}{N+\alpha} > \frac{N}{N-2}, N+\alpha-(N-2)(p+1)\ge 0, (N-2)r-4>0,\\
 \frac{2Np}{N+\alpha} > \frac{N}{N-2}, N+\alpha-(N-2)(p+1)< 0, 2\alpha+(N-2)(r-2p)>0, \\
\frac{2Np}{N+\alpha} = \frac{N}{N-2},4+\alpha-N>0, (N-2)r-4>0, \\
 \frac{2Np}{N+\alpha} = \frac{N}{N-2}, 4+\alpha-N\le 0, \alpha-N+(N-2)r>0, \\
	\frac{2Np}{N+\alpha} < \frac{N}{N-2}, (N-2)r-4>0. \\
\end{cases}
\end{equation*}
\end{proposition}
\begin{proof} Note that
\begin{eqnarray*}
\varphi_{c^*,v_\varepsilon}(t)&=&t^{-s_r}\frac{\left[\frac{S+O(\varepsilon^{N-2})}{2}+\frac{1}{2p}\int |I_{\alpha/2}\star |v_{\varepsilon}|^p|^2\right]t^{s_q}-\frac{t^{s_{2^*}}}{2^*}-c^*}{F(v_\varepsilon)} \\
&=& t^{s_q-s_r }\frac{\frac{O(\varepsilon^{N-2})}{2}+\frac{1}{2p}\int |I_{\alpha/2}\star |v_{\varepsilon}|^p|^2}{F(v_\varepsilon)}+t^{-s_r}\frac{z_{v_\varepsilon}(t)-c^*}{F(v_\varepsilon)}, t>0,\ \mbox{small}\ \varepsilon>0,
\end{eqnarray*}
where $v_\varepsilon$ is given by Lemma \ref{ltalenti}. By Lemma \ref{lt0} we conclude that 

\begin{equation*}
	\varphi_{c^*,v_\varepsilon}(t)\le   t^{s_q-s_r}\frac{\frac{O(\varepsilon^{N-2})}{2}+\frac{1}{2p}\int |I_{\alpha/2}\star |v_{\varepsilon}|^p|^2}{F(v_\varepsilon)}, t>0,\ \mbox{small}\ \varepsilon>0.
\end{equation*}

Since $0\le 	\Lambda_{c^*}((v_{\varepsilon})_{t_{c^*}(v_{\varepsilon})})$, $\varepsilon>0$, by Lemma \ref{ineinfipositive}, then
\begin{equation}\label{talentiine}
0<	\Lambda_{c^*}((v_{\varepsilon})_{t_{c^*}(v_{\varepsilon})})=	\varphi_{c^*,v_\varepsilon}(t_{c^*}(v_{\varepsilon}))\le    t_{c^*}(v_{\varepsilon})^{s_q-s_r }\frac{\frac{O(\varepsilon^{N-2})}{2}+\frac{1}{2p}\int |I_{\alpha/2}\star |v_{\varepsilon}|^p|^2}{F(v_\varepsilon)}, \varepsilon>0. 
\end{equation}
Now we claim that $t_{c^*}(v_{\varepsilon})$ is away from zero. In fact, from $H((v_{\varepsilon})_{t_{c^*}(v_{\varepsilon})})=0$, we conclude that
\begin{equation*}
(s_r-s_q)t_{c^*}(v_{\varepsilon})^{s_q}\left(\frac{S+O(\varepsilon^{N-2})}{2}+\frac{1}{2p}\int |I_{\alpha/2}\star |v_{\varepsilon}|^p|^2\right)+\frac{s_{2^*}-s_r}{2^*}t_{c^*}(v_{\varepsilon})^{s_{2^*}}=s_rc^*, \varepsilon>0,
\end{equation*}
which implies, by Lemma \ref{ltalenti}, the desired claim. Since $t_{c^*}(v_{\varepsilon})$ is away from zero we conclude from \eqref{talentiine} and Lemma \ref{la1} that
\begin{equation*}
\lim_{\varepsilon\to 0}	\Lambda_{c^*}((v_{\varepsilon})_{t_{c^*}(v_{\varepsilon})})=0,
\end{equation*}
and the proof is complete.
\end{proof}

\begin{proposition}\label{lbb2} If $c>c^*$, then $\inf_{u\in \mathcal{N}_{c}}\Lambda_{c}(u)<0$.
\end{proposition}

\begin{proof} Write $t_{\varepsilon}=t_{c}(v_{\varepsilon})$ and note that 
	
	\begin{eqnarray*}
		\varphi_{c^*,v_\varepsilon}(t_{\varepsilon})&=&t_{\varepsilon}^{s_r-s_q }\frac{\frac{O(\varepsilon^{N-2})}{2}+\frac{1}{2p}\int |I_{\alpha/2}\star |v_{\varepsilon}|^p|^2+t_{\varepsilon}^{-s_r}[z_{v_\varepsilon}(t_{\varepsilon})-c]}{F(v_\varepsilon)} \\
		&\le & t_{\varepsilon}^{s_r-s_q }\frac{\frac{O(\varepsilon^{N-2})}{2}+\frac{1}{2p}\int |I_{\alpha/2}\star |v_{\varepsilon}|^p|^2+t_{\varepsilon}^{-s_r}[c^*-c]}{F(v_\varepsilon)},\ \mbox{small}\ \varepsilon>0,
	\end{eqnarray*}
\end{proof}
which implies, by Lemma \ref{ltalenti}, that
\begin{equation*}
	\Lambda_c(t_{\varepsilon}v_\varepsilon)=	\varphi_{c^*,v_\varepsilon}(t_{\varepsilon})<0, \ \mbox{small}\ \varepsilon>0,
\end{equation*}
and the proof is complete.
\subsection{Analysis of the curves $\lambda_{c,k}$}

In this section we will analyze the behavior of $(0,c^*)\ni c\mapsto \lambda_{c,k}$. By Proposition \ref{NEHARI} we know that $\mathcal{M}$ and $\mathcal{N}_c$ are homeomorphic through a odd homemorphism, so it follows from the monotonicity of the index that
\begin{equation} \label{31}
	\lambda_{c,k} = \inf_{M \in \F_k}\, \sup_{u \in M}\, \widetilde{\Lambda}_c(u),
\end{equation}
where $\mathcal{F}_k\in \mathcal{F}$ and $\mathcal{F}$ denote the class of symmetric subsets of $\mathcal{M}$.

For fixed $u \in \M$, let
\[
\widetilde{H}(c,t) =H(u_t)-s_rc, \quad (c,t) \in  (0,c^*)\times(0,\infty)
\]
and note that
\begin{equation} \label{32}
\widetilde{H}(c,t_c(u)) =H(u_{t_c(u)}) -s_rc= 0.
\end{equation}
We have
\[
\frac{\partial \widetilde{H}}{\partial t}(c,t_c(u)) = s_q(s_r-s_q)t_c(u)^{s_q-1}I(u)+s_{2^*}(s_{2^*}-s_r)t_c(u)^{s_{2^*}-1}G(u) \ne 0,
\]
 so it follows from the implicit function theorem that the mapping $(0,c^*) \to (0,\infty),\, c \mapsto t_c(u)$ is $C^1$. Then the mapping $c \mapsto \widetilde{\Lambda}_c(u)$ is also $C^1$ and
\begin{equation} \label{33}
	\frac{\partial \widetilde{\Lambda}_c(u)}{\partial c} = \frac{t_c(u)^{-s_r-1} H(c,t_c(u))\, \dfrac{\partial t_c(u)}{\partial c} - t_c(u)^{-s_r}}{F(u)} = - \frac{t_c(u)^{-s_r}}{F(u)}
\end{equation}
by \eqref{32}. First we note that the curves $C_{k}:=\{(\lambda_{c,k},c): c\in (0,c^*)\}$ are nonincreasing.

\begin{lemma} \label{Lemma 1}
	The following hold.
	\begin{enumroman}
		\item For each $u \in \M$, the mapping $(0,c^*] \to \R,\, c \mapsto \widetilde{\Lambda}_c(u)$ is decreasing.
		\item For each $k \ge 1$, the mapping $(0,c^*] \to \R,\, c \mapsto \lambda_{c,k}$ is nonincreasing.
	\end{enumroman}
\end{lemma}

\begin{proof}
	$(i)$ holds since $\dfrac{\partial \widetilde{\Lambda}_c(u)}{\partial c} < 0$ by
	\eqref{33}. $(ii)$ follows from $(i)$ and \eqref{31}.
\end{proof}

Lemma \ref{Lemma 1} implies that we can work in a suitable sublevel set $\widetilde{\Lambda}_c^T = \bgset{u \in \M : \widetilde{\Lambda}_c(u) \le T}$ in \eqref{31}.

\begin{lemma} \label{Lemma 2}
	Let $[a,b] \subset (0,c^*]$, let $k \ge 1$, and let $T > \lambda_{a,k}$. Denote by $\F_{b,T}$ the class of symmetric subsets of $\widetilde{\Lambda}_b^T$ and let $\F_{b,T,k} = \bgset{M \in \F_{b,T} : i(M) \ge k}$. Then
	\[
	\inf_{M \in \F_{b,T,k}}\, \sup_{u \in M}\, \widetilde{\Lambda}_c(u) = \lambda_{c,k} \quad \forall c \in [a,b].
	\]
\end{lemma}

\begin{proof}
	Clearly, $\F_{b,T,k} \subset \F_k$. If $M \in \F_k \setminus \F_{b,T,k}$, then
	\[
	\sup_{u \in M}\, \widetilde{\Lambda}_c(u) \ge \sup_{u \in M}\, \widetilde{\Lambda}_b(u) > T >	\lambda_{a,k} \ge \lambda_{c,k}
	\]
	by Lemma \ref{Lemma 1}.
\end{proof}
The following inequality will prove useful:
\begin{equation}\label{ineqnorms}
		\min\set{t_c(u)^{\frac{s_q}{2}},t_c(u)^{\frac{s_q}{2p}}} \norm{u}\le\norm{u_{t_c(u)}} \le \max \set{t_c(u)^{\frac{s_q}{2}},t_c(u)^{\frac{s_q}{2p}}} \norm{u}, u\in E_R^{\alpha,p}(\mathbb{R}^N)\setminus\{0\}, c>0.
\end{equation}
\begin{lemma} \label{Lemma 3}
	Let $[a,b] \subset (0,c^*)$ and let $T \in \R$. Then $\exists C > 0$ such that the following hold.
	\begin{enumroman}
		\item $F(u) \ge C^{-1} \quad \forall u \in \widetilde{\Lambda}_b^T$
		\item $C^{-1} \le t_c(u) \le C \quad \forall (u,c) \in \widetilde{\Lambda}_b^T \times [a,b]$
		\item $- C \le \dfrac{\partial \widetilde{\Lambda}_c(u)}{\partial c} \le - C^{-1} \quad \forall (u,c) \in \widetilde{\Lambda}_b^T \times [a,b]$
	\end{enumroman}
\end{lemma}

\begin{proof}
	$(i)$ On the contrary, we can find a sequence $u_n\in \widetilde{\Lambda}_b^T$ such that $F(u_n)\to 0$ as $n\to \infty$. We claim that $F((u_n)_{t_b(u_n)})\to 0$ as $n\to \infty$.  Indeed, Lemma \ref{lbound} together with $u_n\in \mathcal{M}$ and inequality \eqref{ineqnorms} implies that $t_b(u_n)$ is bounded and so $F((u_n)_{t_b(u_n)})=t_b(u_n)^{s_r}F(u_n)\to 0$ as $n\to \infty$. Therefore, by Corollary \ref{cinfinity}, we conclude that $\widetilde{\Lambda}_b(u_n)\to \infty$, a contradiction.

	$(ii)$ This follow from Lemma \ref{lbound}, $u\in \mathcal{M}$ and inequality \eqref{ineqnorms}.

	$(iii)$ This is immediate from \eqref{33}, $(i)$, and $(ii)$.
\end{proof}

\begin{proposition} \label{Proposition 1}
	For each $k \ge 1$, the mapping $(0,c^*) \to \R,\, c \mapsto \lambda_{c,k}$ is continuous and decreasing.
\end{proposition}

\begin{proof}
	Let $[a,b] \subset (0,c^*)$ and let $T > \lambda_{a,k}$. Then
	\begin{equation} \label{39}
		\inf_{M \in \F_{b,T,k}}\, \sup_{u \in M}\, \widetilde{\Lambda}_c(u) = \lambda_{c,k} \quad \forall c \in [a,b]
	\end{equation}
	by Lemma \ref{Lemma 2}. For each $u \in \widetilde{\Lambda}_b^T$,
	\[
	\widetilde{\Lambda}_b(u) - \widetilde{\Lambda}_a(u) = \frac{\partial \widetilde{\Lambda}_c(u)}{\partial c}\, (b - a)
	\]
	for some $c \in (a,b)$ by the mean value theorem and hence
	\[
	- C\, (b - a) \le \widetilde{\Lambda}_b(u) - \widetilde{\Lambda}_a(u) \le - C^{-1}\, (b - a)
	\]
	by Lemma \ref{Lemma 3} $(iii)$. This together with \eqref{39} gives
	\[
	- C\, (b - a) \le \lambda_{b,k} - \lambda_{a,k} \le - C^{-1}\, (b - a),
	\]
	from which the desired conclusions follow.
\end{proof}

Now we study the behavior of $\lambda_{c,k}$ when $c\to 0^+$.
\begin{lemma}\label{lc0} If $r>q$, then 
	\begin{enumerate}
		\item[(i)] $\lim_{c\to 0^+}\sup_{u\in \mathcal{M}}t_c(u)=0$.
		\item[(ii)] $\lim_{c\to 0^+}\inf_{u\in \mathcal{M}}t_c(u)^{-s_r}c=\infty$.
	\end{enumerate}
\end{lemma}
\begin{proof} Suppose $r>q$. From $H(u_{t_c(u)})=s_rc$ we conclude that 
	\begin{equation}\label{101}
		(s_r-s_q)t_c(u)^{s_q}+(s_{2^*}-s_r)t_c(u)^{s_{2^*}}G(u)=s_rc, u\in \mathcal{M},
	\end{equation}
	hence
	\begin{equation*}
		t_c(u)\le \left[\frac{s_rc}{s_r-s_q}\right]^{\frac{1}{s_q}}, u\in \mathcal{M},
	\end{equation*}
	which implies $(i)$. To prove $(ii)$ note from \eqref{101} that
	\begin{equation*}
	(s_r-s_q)t_c(u)^{s_q-s_r}+(s_{2^*}-s_r)t_c(u)^{s_{2^*}-s_r}G(u)=s_rt_c(u)^{-s_r}c, u\in \mathcal{M},
	\end{equation*}
	therefore
	\begin{equation*}
		\lim_{c\to 0^+}t_c(u)^{-s_r}c=	\lim_{c\to 0^+}\frac{s_r-s_q}{s_r}t_c(u)^{s_q-s_r}=\infty,\ \mbox{uniformly}\ u\in \mathcal{M},
	\end{equation*}
	by item $(i)$.

\end{proof}
\begin{corollary} \label{cunif} If $r>q$, then $\lim_{c\to 0^+}\inf_{u\in \mathcal{M}}\widetilde{\Lambda}_c(u)=\infty$.
\end{corollary}
\begin{proof} Indeed, by \eqref{101} we have that
	\begin{equation*}
		\widetilde{\Lambda}_c(u)=\frac{s_qt_c(u)^{-s_r}c-(s_{2^*}-s_q)t_c(u)^{s_{2^*}-s_r}G(u)}{(s_r-s_q)F(u)}, u\in \mathcal{M}.
	\end{equation*}
	Since $F(u),G(u)$ are bounded over $\mathcal{M}$ the proof follows by Lemma \ref{lc0}.
\end{proof}

Before the study of $\lim_{c\to 0^+}\lambda_{c,k}$, let us introduce a sequence of scaled eigenvalues to the problem 
\begin{equation}\label{1}
	I'(u)=\lambda F'(u).
\end{equation}
 If $r=q$, write

\begin{equation} \label{5}
	\lambda_k := \inf_{M \in \F_k}\, \sup_{u \in M}\, \widetilde{\Psi}(u),
\end{equation}
where $\widetilde{\Psi}(u)=1/F(u)$. The following theorem was proved in \cite{MePe2}.

\begin{theorem}[{\cite[Theorem 2.10]{MePe2}}] \label{Theorem 7}
$\lambda_k \nearrow \infty$ is a sequence of eigenvalues of problem \eqref{1}.
	\begin{enumroman}
		\item The first eigenvalue is given by
		\[
		\lambda_1 = \min_{u \in \M}\, \widetilde{\Psi}(u) > 0.
		\]
		\item If $\lambda_k = \dotsb = \lambda_{k+m-1} = \lambda$ and $E_\lambda$ is the set of eigenfunctions associated with $\lambda$ that lie on $\M$, then $i(E_\lambda) \ge m$.
		\item If $\lambda_k < \lambda < \lambda_{k+1}$, then
		\[
		i(\widetilde{\Psi}^{\lambda_k}) = i(\M \setminus \widetilde{\Psi}_\lambda) = i(\widetilde{\Psi}^\lambda) = i(\M \setminus \widetilde{\Psi}_{\lambda_{k+1}}) = k,
		\]
		where $\widetilde{\Psi}^a = \bgset{u \in \M : \widetilde{\Psi}(u) \le a}$ and $\widetilde{\Psi}_a = \bgset{u \in \M : \widetilde{\Psi}(u) \ge a}$ for $a \in \R$.
	\end{enumroman}
\end{theorem}
Now we have the tools to analyze $\lim_{c\to 0^+}\lambda_{c,k}$:
\begin{proposition} \label{Proposition 2} For each $k\in \mathbb{N}$ we have
	\begin{enumerate}
		\item[(i)] If $r>q$, then $\lim_{c\to 0^+}\lambda_{c,k}=\infty$.
		\item[(ii)] 	If $r=q$, then 
		\begin{equation*}
			\lim_{c\to 0^+}\lambda_{c,k}=\lambda_k, \forall k\in \mathbb{N},
		\end{equation*}
	\end{enumerate} 
\end{proposition}

\begin{proof} $(i)$ In fact, since $\lambda_{c,k}\ge \inf_{u\in \mathcal{M}_s}\widetilde{\Lambda}_c(u)$, the proof follows from Corollary \ref{cunif}. $(ii)$ The proof is similar to that of \cite[Proposition 2.18 item $(ii)$]{PeSi}.

\end{proof}
Now we study we behavior of $\lambda_{c,k}$ when $c\to (c^*)^-$. 
\begin{proposition}\label{Proposition 4} For each $k\in \mathbb{N}$, there exists $\tilde{\lambda}_k\ge 0$ such that 
	\begin{enumerate}
		\item[(i)] 	$\lim_{c\to (c^*)^-}\lambda_{c,k}=\tilde{\lambda}_{k}$.
		\item[(ii)] $\tilde{\lambda}_k\le \tilde{\lambda}_{k+1}$, $k\in \mathbb{N}$.
		\item[(iii)] $\tilde{\lambda}_1=0$ in the following cases:
			\begin{equation*}
			\begin{cases}
			\frac{2Np}{N+\alpha} > \frac{N}{N-2}, N+\alpha-(N-2)(p+1)\ge 0, (N-2)r-4>0,\\
			\frac{2Np}{N+\alpha} > \frac{N}{N-2}, N+\alpha-(N-2)(p+1)< 0, 2\alpha+(N-2)(r-2p)>0, \\
			\frac{2Np}{N+\alpha} = \frac{N}{N-2},4+\alpha-N>0, (N-2)r-4>0, \\
			\frac{2Np}{N+\alpha} = \frac{N}{N-2}, 4+\alpha-N\le 0, \alpha-N+(N-2)r>0, \\
			\frac{2Np}{N+\alpha} < \frac{N}{N-2}, (N-2)r-4>0. \\
		\end{cases}
		\end{equation*}
	\end{enumerate}

\end{proposition}

\begin{proof} $(i)$ It follows from Lemma \ref{Lemma 1}. 
	
	$(ii)$ Since $\lambda_{c,k}\le \lambda_{c,k+1}$ for all $c>0$ and $k\in \mathbb{N}$ it follows that  $\tilde{\lambda}_k\le \tilde{\lambda}_{k+1}$, $k\in \mathbb{N}$.
	
	$(iii)$	Arguing as in the proof of Lemma \ref{lbb} we have that 
\begin{equation*}
	\varphi_{c,v_\varepsilon}(t_{c}(v_{\varepsilon}))\le    t_{c}(v_{\varepsilon})^{s_q-s_r}\frac{\frac{O(\varepsilon^{N-2})}{2}+\frac{1}{2p}\int |I_{\alpha/2}\star |v_{\varepsilon}|^p|^2}{F(v_\varepsilon)}+ t_{c}(v_{\varepsilon})^{-s_r}\frac{c^*-c}{F(v_\varepsilon)}, \varepsilon>0. 
\end{equation*}
	By taking $c=c(\varepsilon):=c^*-\varepsilon F(v_\varepsilon)$ we conclude that 
	\begin{equation*}
			\varphi_{{c(\varepsilon)},v_\varepsilon}(t_{c(\varepsilon)}(v_\varepsilon))\le    t_{c}(v_{\varepsilon})^{s_q-s_r }\frac{\frac{O(\varepsilon^{N-2})}{2}+\frac{1}{2p}\int |I_{\alpha/2}\star |v_{\varepsilon}|^p|^2}{F(v_\varepsilon)}+ t_{c}(v_{\varepsilon})^{-s_r}\varepsilon, \varepsilon>0. 
	\end{equation*}
	Again, similar to the proof of Lemma \ref{lbb} we have that $ t_{c(\varepsilon)}(v_\varepsilon)$ is away from zero and hence, by using Lemmas \ref{ineinfipositive} and \ref{la1}, we conclude that
	\begin{equation*}
		0\le 	\lim_{\varepsilon\to 0}	\Lambda_{c(\varepsilon)}(t_{c(\varepsilon)}(v_\varepsilon)v_\varepsilon)=\lim_{\varepsilon\to 0}	\varphi_{{c(\varepsilon)},v_\varepsilon}(t_{c(\varepsilon)}(v_\varepsilon))=0
	\end{equation*}
	which completes the proof.
	
\end{proof}

\section{Proof of Theorem \ref{thm1}}

Define $L_{\lambda}:=\{(\lambda,c): c\in (0,c^*)\}$ and for $k\in \mathbb{N}$ let $C_{k}=\{(\lambda_{c,k},c): c\in (0,c^*)\}$. Recall that $\tilde{\lambda}_k$ was defined in Proposition \ref{Proposition 4}. First, combining the results of the last section, we have the following 
\begin{theorem} \label{thm2} If $r>q$, then 
	\begin{enumerate}
		\item[(i)] For each $k\in \mathbb{N}$, the curves $(0,c^*)\ni c\mapsto \lambda_{c,k}$ are continuous and decreasing.
		\item[(ii)] For each $k\in \mathbb{N}$ we have that $\lim_{c\to 0^+}=\infty$.
		\item[(iii)] For each $k\in \mathbb{N}$ we have that $\lim_{c\to (c^*)^-}\lambda _{c,k}=\tilde{\lambda}_k$.
		\item[(iv)] For each $k\in \mathbb{N}$ and $\lambda>\tilde{\lambda}_k$, the set $L_{\lambda}\cap C_{k}$ has at least $k$ points. 
	\end{enumerate}
	
\end{theorem}
\begin{proof} $(i)$ follows from Proposition \ref{Proposition 1}, $(ii)$ follows from Proposition \ref{Proposition 2} and $(iii)$ from Proposition \ref{Proposition 4}. $(iv)$ is clear from $(i)-(iii)$.
\end{proof}
\begin{figure}[h!]
	\centering
	\begin{tikzpicture}[>=latex]
		\draw[->] (-1,0) -- (6,0) node[below] {\scalebox{0.8}{$\lambda$}};
		\foreach \x in {}
		\draw[shift={(\x,0)}] (0pt,2pt) -- (0pt,-2pt) node[below] {\footnotesize $\x$};
		\draw[->] (0,-1) -- (0,3) node[left] {\scalebox{0.8}{$c$}};
		\foreach \y in {}
		\draw[shift={(0,\y)}] (2pt,0pt) -- (-2pt,0pt) nde[left] {\footnotesize $\y$};
		\draw[blue,thick] (1.55,2.5) .. controls (2,.2) and (3,0.2) .. (6,.2);
		\draw [thick] (1.5,-.1) node[below]{\scalebox{0.8}{$\tilde{\lambda}_{1}$}} -- (1.5,0.05); 
		\draw [dashed] (1.5,0) -- (1.5,2.5);
		\draw [thick] (2.3,-.1) node[below]{\scalebox{0.8}{$\tilde{\lambda}_{2}$}} -- (2.3,0.05); 
		\draw [dashed] (2.3,0) -- (2.3,2.5);
		\draw [thick] (3.3,-.1) node[below]{\scalebox{0.8}{$\tilde{\lambda}_{k}$}} -- (3.3,0.05); 
		\draw [dashed] (3.3,0) -- (3.3,2.5);
		
		\draw [thick] (4,.5) node[above]{$\cdots$}; 
		\draw[blue,thick] (2.35,2.5) .. controls (2.8,.4) and (3.8,0.4) .. (6,.4);
		\draw[blue,thick] (3.35,2.5) .. controls (3.8,.6) and (4.8,0.6) .. (6,.6);
		
		\draw [thick] (5,0) -- (5,2.5);
		\draw (5,.64) node{\scalebox{0.8}{$\bullet$}};
		\draw (5,.4) node{\scalebox{0.8}{$\bullet$}};
		\draw (5,.2) node{\scalebox{0.8}{$\bullet$}};
		\draw  (1,1) node[below]{\scalebox{1.5}{$\nexists$}}  ; 	
		\draw  (1.5,2.5) node[above]{\scalebox{0.8}{$\lambda_{c,1}$}} ; 	
		\draw  (2.3,2.5) node[above]{\scalebox{0.8}{$\lambda_{c,2}$}} ; 
		
		\draw  (3.3,2.5) node[above]{\scalebox{0.8}{$\lambda_{c,k}$}} ; 
		
		\draw  (5,0) node[below]{\scalebox{0.8}{$L_\lambda$}} ;
		\draw [dashed] (0,2.5) -- (6,2.5);
		\draw  (0,2.5) node[left]{\scalebox{0.8}{$c^*$}} ;
	\end{tikzpicture}
	\caption{Energy curves Theorem \ref{thm2}} \label{fig}
\end{figure}
\begin{theorem} \label{thm3} If $r=q$ and $\alpha>1$, then 
	\begin{enumerate}
		\item[(i)] For each $k\in \mathbb{N}$, the curves $(0,c^*)\ni c\mapsto \lambda_{c,k}$ are continuous and decreasing.
		\item[(ii)] For each $k\in \mathbb{N}$ we have that $\lim_{c\to 0^+}=\lambda_k$.
		\item[(iii)] For each $k\in \mathbb{N}$ we have that $\lim_{c\to (c^*)^-}\lambda _{c,k}=\tilde{\lambda}_k$.
	\end{enumerate}
	
\end{theorem}
\begin{proof} $(i)$ follows from Proposition \ref{Proposition 1}, $(ii)$ follows from Proposition \ref{Proposition 2} and $(iii)$ from Proposition \ref{Proposition 4}. 
\end{proof}
\begin{figure}[h!]
	\centering
	\begin{tikzpicture}[>=latex]
		\draw[->] (-1,0) -- (7,0) node[below] {\scalebox{0.8}{$\lambda$}};
		\foreach \x in {}
		\draw[shift={(\x,0)}] (0pt,2pt) -- (0pt,-2pt) node[below] {\footnotesize $\x$};
		\draw[->] (0,-1) -- (0,3) node[left] {\scalebox{0.8}{$c$}};
		\foreach \y in {}
		\draw[shift={(0,\y)}] (2pt,0pt) -- (-2pt,0pt) nde[left] {\footnotesize $\y$};
		\draw[blue,thick] (1.55,2.5) .. controls (2,.2) and (3,0) .. (4.1,0);
		\draw [thick] (1.5,-.1) node[below]{\scalebox{0.8}{$\tilde{\lambda}_{1}$}} -- (1.5,0.05); 
		\draw [dashed] (1.5,0) -- (1.5,2.5);
		\draw [thick] (2.3,-.1) node[below]{\scalebox{0.8}{$\tilde{\lambda}_{2}$}} -- (2.3,0.05); 
		\draw [dashed] (2.3,0) -- (2.3,2.5);
		\draw [thick] (3.3,-.1) node[below]{\scalebox{0.8}{$\tilde{\lambda}_{k}$}} -- (3.3,0.05); 
		\draw [dashed] (3.3,0) -- (3.3,2.5);
		
				\draw [thick] (4.1,-.1) node[below]{\scalebox{0.8}{$\lambda_{1}$}} -- (4.1,0.05); 
	
		\draw [thick] (4.9,-.1) node[below]{\scalebox{0.8}{$\lambda_{2}$}} -- (4.9,0.05); 
	
		\draw [thick] (5.7,-.1) node[below]{\scalebox{0.8}{$\lambda_{k}$}} -- (5.7,0.05); 
	
		\draw [thick] (3.8,.5) node[above]{$\cdots$}; 
			\draw [thick] (6,.5) node[above]{$\cdots$}; 
		\draw[blue,thick] (2.35,2.5) .. controls (2.8,.4) and (3.8,0) .. (4.9,0);
		\draw[blue,thick] (3.35,2.5) .. controls (3.8,.6) and (4.8,0) .. (5.7,0);
		
		\draw  (1,1) node[below]{\scalebox{1.5}{$\nexists$}}  ; 	
		\draw  (1.5,2.5) node[above]{\scalebox{0.8}{$\lambda_{c,1}$}} ; 	
		\draw  (2.3,2.5) node[above]{\scalebox{0.8}{$\lambda_{c,2}$}} ; 
		
		\draw  (3.3,2.5) node[above]{\scalebox{0.8}{$\lambda_{c,k}$}} ; 
		
		\draw [dashed] (0,2.5) -- (6,2.5);
		\draw  (0,2.5) node[left]{\scalebox{0.8}{$c^*$}} ;
	\end{tikzpicture}
	\caption{Energy curves Theorem \ref{thm3}} \label{fig1}
\end{figure}

Now we prove Theorem \ref{thm1}.

\begin{proof}[Proof of Theorem \ref{thm1}] 
	
	$(i)$ By Lemma \ref{Lemma 8}, if there exists $u\in E_R^{\alpha,p}(\mathbb{R}^N) $ such that $\Phi_\lambda(u)=c$ and $\Phi_\lambda'(u)=0$, then $u\in \mathcal{N}_c$, however, Proposition \ref{NEHARI} implies that  $ \mathcal{N}_c=\empty$ if $c\le 0$, therefore equation \eqref{103} has no nontrivial solution with energy $c \le 0$. The existence of $\lambda_{c,k}$, for all $c>0$, is given by Propositions \ref{Proposition 3} and \ref{Proposition 7}.
	
	$(ii)$ This follows from Theorem \ref{thm2} item $(iv)$ and Proposition \ref{Proposition 4}. 
	
	$(iii)$ This is a consequence of Theorem \ref{thm3}.
	
	$(iv)$ This follows from Proposition \ref{Proposition 4}.

\end{proof}

\begin{remark}\label{rnexistence} \noindent
	\begin{enumerate}
		\item[(i)] $\lambda_{c^*,k}\le \tilde{\lambda}_k$ for all $k\in \mathbb{N}$.
		\item[(ii)] Note from Proposition \ref{lbb2} that $\inf_{u\in \mathcal{N}_c}\Lambda_c(u)<0$ when $c>c^*$ and, since equation \eqref{103} has no positive solutions for $\lambda\le 0$ (by Pohozaev's identity), we must conclude that equation \eqref{103} has no positive solutions for $\lambda< \lambda_{c^*,1}$. 
		\item[(iii)] We believe that $\lambda_{c^*,1}>0$ when \eqref{cases} are not satisfied.
		\item[(iv)] Proposition \ref{lbb1} together with Proposition \ref{Proposition 4} item $(iii)$ imply that $(0,c^*]\ni c\mapsto \lambda_{c,1}$ is continuous when \eqref{cases} is satisfied.
	\end{enumerate}

\end{remark}

\newpage
{\bf Acknowledgement.}
This work was completed while the second author held a post-doctoral position at Florida Institute of Technology, Melbourne, United States of
America, supported  by CNPq/Brazil under Grant 201334/2024-0.

\def\cprime{$''$}

\end{document}